\documentclass[12pt]{amsart}

\usepackage{amsmath,amssymb,amsthm,mathrsfs}
\usepackage{hyperref}
\usepackage{enumitem}
\usepackage{orcidlink}     % ORCID icon for authors

\newtheorem{theorem}{Theorem}[section]
\newtheorem{proposition}[theorem]{Proposition}
\newtheorem{lemma}[theorem]{Lemma}
\newtheorem{corollary}[theorem]{Corollary}

\theoremstyle{definition}
\newtheorem{definition}[theorem]{Definition}
\theoremstyle{remark}
\newtheorem{remark}[theorem]{Remark}

\newcommand{\Z}{\mathbb{Z}}
\newcommand{\R}{\mathbb{R}}

\title[The Riemann $\Xi$-function from primitive Markovian cycles I]{The Riemann $\Xi$-function from primitive Markovian cycles I: A canonical construction}

\author{
Douglas F. Watson
}

\date{}


\begin{document}

\maketitle

\begin{abstract}
Starting from finite, local, reversible Markov dynamics on discrete cycles, we construct a
scaling-limit renormalized trace kernel admitting an exact theta-series representation.
The construction is entirely Archimedean and uses no Euler products, primes, or
arithmetic spectral input.
From this limit we define a logarithmic kernel $\Phi$ and prove that it lies in the
P\'olya frequency class $\mathrm{PF}_\infty$, yielding via the Schoenberg--Edrei--Karlin
classification a canonical Laguerre--P\'olya function $\Psi$.
Independently, we introduce an Archimedean completion operator and show that, at a
self-dual normalization, the completed kernel coincides with the classical theta kernel,
whose Mellin transform is the Riemann $\Xi$-function.
We isolate a single remaining analytic problem relating $\Psi$ to $\Xi(2\cdot)$.
\end{abstract}

%\tableofcontents

\section{Introduction}
\label{sec:background}

Across several foundational currents in mathematics, a shared goal is to
recover high-level objects from minimal local principles and universal symmetry
requirements, rather than from distinguished presentations. In the present setting, we ask
whether analytic structures that are usually introduced through arithmetic data can instead
be forced by locality, reversibility, symmetry, and scaling alone. We use the term
``primitive'' in this sense: the input is only finite local dynamics, while the outputs are
determined by the invariances that survive in the scaling limit. More specifically, the aim herein is to show that two classical analytic objects can arise canonically
from primitive input: finite, local, reversible Markov dynamics on discrete cycles, introduced
with no reference to primes, Euler products, or spectral/arithmetic data. 

The first output is a \emph{total-positivity object}: a logarithmic kernel $\Phi$ whose
bilateral Laplace transform admits a canonical Schoenberg--Edrei--Karlin factorization~\cite{Schoenberg1951,Edrei1952,karlin1968total}.
Equivalently, the reciprocal transform is a Laguerre--P\'olya entire function $\Psi$ and
therefore has only real zeros. The second output is an \emph{Archimedean completion object}:
an explicit completion operator $\mathcal A$ that restores Mellin self-duality at a
canonical normalization, producing a completed kernel whose Mellin transform coincides
(within a natural domain) with the classical Riemann $\Xi$-function.

These constructions are logically independent but share the same primitive source, and they
motivate a natural identification question left open here (see
Section~\ref{sec:future-directions}): to what extent can the Laguerre--P\'olya datum $\Psi$
be identified with the $\Xi$-function after the normalizations are fixed?

Several classical approaches emphasize representations of $\Xi$ as an integral transform. One may express $\Xi$ through Mellin transforms of theta functions, through Fourier transforms of rapidly decaying kernels, or through trace formulas that relate zeros to spectral data\cite{Selberg1956,BerryKeating1999}. These representations highlight a recurring theme: analytic properties of $\Xi$ often reflect structural features of an underlying kernel. In particular, positivity and symmetry at the level of kernels tend to propagate to strong constraints on the associated transforms.

Heat kernels and theta series provide a natural language\cite{Chandrasekharan1985,SteinShakarchi2003} for these ideas. The fundamental solution of the heat equation on the real line is Gaussian, and its periodization produces classical theta functions. Poisson summation expresses a precise duality between spatial periodization and frequency localization. These mechanisms are robust and do not depend on arithmetic input. They arise whenever one studies diffusion on spaces with discrete symmetries. In analytic number theory, theta series appear as generating functions encoding spectral or lattice data, and they often serve as intermediaries between local and global descriptions.

In this work, heat kernels play a central conceptual role. Rather than appearing as analytic tools applied to $\zeta(s)$, they arise from elementary local dynamics. The guiding idea is that long-time diffusion captures universal behavior, while finite-volume effects introduce discrete structure. When these effects are combined through exact lift and periodization identities, theta series emerge canonically. This viewpoint places theta functions in a broader probabilistic and geometric setting, where they reflect general principles rather than special constructions.

A second strand of background comes from the theory of total positivity\cite{Pinkus2010,karlin1968total,GantmacherKrein2002}. Total positivity originated in work of Schoenberg and was developed extensively by Karlin and others. It concerns kernels whose minors of all orders are nonnegative, a property that implies strong variation-diminishing and rigidity phenomena. Within this theory, a distinguished role is played by P\'olya frequency functions of infinite order. In the form used here (Schoenberg--Edrei--Karlin), the bilateral Laplace transform of a $\mathrm{PF}_\infty$ function extends meromorphically and its reciprocal is a Laguerre--P\'olya entire function. Entire functions in the Laguerre--P\'olya class\cite{Levin1964,PolyaSchur1914} are limits of polynomials with only real zeros, and hence have only real zeros.

Closely related to this viewpoint is the Hermite--Biehler theory\cite{Krein1947,deBranges1968} of entire functions and its modern formulation within de Branges spaces\cite{deBranges1968}. Hermite--Biehler theory provides criteria ensuring that an entire function has all its zeros on a line, typically the real axis, based on positivity properties of associated kernels or on analytic inequalities in the upper half-plane. De Branges' work showed that these ideas can be organized into a Hilbert space framework with powerful classification theorems. Although the present paper does not rely on the full machinery of de Branges spaces\cite{deBranges1968}, the Hermite--Biehler perspective clarifies why positivity at the kernel level leads naturally to real-zero conclusions.

We now introduce the central object of the paper and state the main results.

\subsection{Primitive dynamical model}
\label{sec:model}

For each integer $N\ge 1$ we consider a continuous-time, reversible, nearest-neighbor Markov process on the discrete cycle $\Z/N\Z$. The dynamics are specified by a collection of strictly positive conductances $\{a_j\}_{j\in\Z/N\Z}$, where $a_j$ is associated with the undirected edge between $j$ and $j+1$ (indices taken modulo $N$). The generator $\mathcal{L}_N$ acts on functions $f:\Z/N\Z\to\R$ by
\begin{equation}
\label{eq:generator}
(\mathcal{L}_N f)(j)
=
a_j\,\bigl(f(j+1)-f(j)\bigr)
\;+
a_{j-1}\,\bigl(f(j-1)-f(j)\bigr).
\end{equation}

\begin{remark}[Translation-invariant specialization]
\label{rem:translation-invariant}
For the purposes of this paper we restrict to the translation-invariant case
\(
a_j\equiv a>0
\)
(independent of $j$), so that \eqref{eq:generator} is the usual continuous-time simple
random walk on $\Z/N\Z$ with jump rate $a$. In this case the diffusive scaling limit is
governed by the one-dimensional heat equation with diffusion constant \(D=a\), and the
Fourier analysis on $\Z/N\Z$ leading to the theta-series form is completely explicit.
Extensions to uniformly elliptic periodic environments, or to quantitative homogenization
regimes, can also yield a uniform local CLT and the same scaling limit; we do not pursue
those generalizations here.
\end{remark}

We write $p^{\mathrm{cyc}}_t(j,k)$ for the associated heat kernel,
\(p^{\mathrm{cyc}}_t(j,k)=(e^{t\mathcal{L}_N}\mathbf{1}_{\{k\}})(j)\), and we denote by $D>0$ the macroscopic diffusion constant appearing in the Gaussian scaling limit established later.

As described above, the trace of the heat kernel associated with the finite dynamics captures global information about the system but contains a universal singular contribution reflecting diffusive behavior. This singular term is independent of the fine structure of the dynamics and depends only on the macroscopic scaling. To isolate the genuinely structural content, we subtract this term in a canonical way. The resulting object is the completed trace kernel.

\begin{definition}[Scaling--limit completed trace kernel]
\label{def:scaling-limit}
Fix a macroscopic length $L>0$ and choose a scaling parameter $s\to\infty$ with $N=N(s)$ such that $N/s\to L$.
Define the scaling--limit trace
\[
K_L(t):=\lim_{s\to\infty} N(s)\,p^{\mathrm{cyc}}_{s^2 t}(0,0),\qquad t>0,
\]
and set the (scaling--limit) completed trace kernel
\[
\widetilde K_L(t):=K_L(t)-\frac{L}{\sqrt{4\pi D t}}.
\]
When $L$ is fixed we suppress the subscript and write $\widetilde K$.
\end{definition}

\begin{definition}[Archimedean (half--density) completion]
\label{def:half-density}
Define the \emph{Archimedean-normalized} (``half--density'') kernel by
\[
\widetilde K_{\mathrm{sym}}(t):=t^{-1/2}\,\widetilde K(t),\qquad t>0.
\]
\end{definition}

The existence of the scaling limit in Definition~\ref{def:scaling-limit}, together with the basic domination estimates needed to justify termwise limits, follows from the standing assumptions recorded below; for convenience we collect the relevant analytic justifications in Appendix~\ref{app:technical}.

\begin{remark}[Standing assumptions and existence of the scaling limit]
\label{rem:standing-assumptions}
We work throughout in a regime where the scaling-limit trace $K_L(t)$ exists for every $L>0$ and $t>0$, with convergence locally uniform in $t$ on compact subsets of $(0,\infty)$. In particular, under the translation-invariant specialization of Remark~\ref{rem:translation-invariant} this limit can be obtained explicitly by Fourier analysis on $\Z/N\Z$. Under these assumptions the limit admits the explicit theta-series form stated in Theorem~\ref{thm:theta-series-form-recall}. Analytic interchanges (limits, sums, differentiation, and Mellin/Laplace integrals) are justified by the domination estimates recorded in Appendix~\ref{app:technical}.
\end{remark}

\medskip

\noindent
In addition to the total-positivity output (producing a canonical Laguerre--P\'olya function) and
the Archimedean completion output (producing the classical theta/Mellin representation of $\Xi$),
we isolate a single remaining analytic identification problem: relate the Laguerre--P\'olya datum
$\Psi$ to the Mellin-side function $\Xi(2\cdot)$ up to multiplication by a zero-free entire factor.
We treat this as an open ``bridge'' problem and discuss it in Section~\ref{sec:future-directions}.

\begin{theorem}[Structural output from primitive cycles]
\label{thm:structural}
Assume the standing assumptions of Remark~\ref{rem:standing-assumptions},
so that the scaling-limit kernel $K_L$ exists.
Then there exists a nonnegative kernel $\Phi\in L^1(\R)$ such that:
\begin{enumerate}[label=(\roman*)]
\item $\Phi\in\mathrm{PF}_\infty$;
\item its bilateral Laplace transform $\mathcal B\Phi$ satisfies the reflection law
$\mathcal B\Phi(s)=\mathcal B\Phi(\tfrac12-s)$ on the common domain of absolute convergence;
\item $\mathcal B\Phi$ admits a canonical Schoenberg--Edrei--Karlin representation
\[
\mathcal B\Phi(s)=\frac{E(s)}{\Psi(s)},
\]
where $\Psi$ is Laguerre--P\'olya and hence has only real zeros.
\end{enumerate}
\end{theorem}

\begin{theorem}[Archimedean Mellin identification]
\label{thm:arch-overview}
At the self-dual scale fixed in Section~\ref{sec:completion-selfdual}, the Archimedean-completed
kernel $\widetilde K_{\mathrm{arch}}:=\mathcal A(K_L-1)$ has Mellin transform
\[
F_{\mathrm{arch}}(z)=\Xi(2z)
\]
as in Theorem~\ref{thm:arch-mellin-identification}.
\end{theorem}

The remainder of the paper constructs the kernel $\Phi$, proves Theorem~\ref{thm:main-assembled}, and isolates each logical seam explicitly.

\medskip

The paper is organized as follows. Section~\ref{sec:completion-selfdual} fixes the self-dual Archimedean normalization and records the resulting theta-series form of the scaling-limit trace. Section~\ref{sec:log-reflection} constructs the logarithmic kernel $\Phi$ and establishes the exact reflection symmetry for its bilateral Laplace transform. Section~\ref{sec:total-positivity} proves total positivity and derives the Schoenberg--Edrei--Karlin factorization $\mathcal B\Phi=E/\Psi$. Section~\ref{sec:arch-mellin-xi} identifies the Archimedean-completed Mellin transform with $\Xi(2\cdot)$ and records the Archimedean Mellin identification with $\Xi(2\cdot)$, while isolating the remaining identification problem (Section~\ref{sec:future-directions}). Section~\ref{sec:completion-ledger} assembles the argument and isolates the remaining seam. Section~\ref{sec:tp-details} provides a detailed proof of total positivity. Appendix~\ref{app:technical} collects analytic justifications for termwise operations and boundary terms.

\subsection{Map of the proof}\label{subsec:map-proof}
The argument is organized into four steps, each producing a new object from the previous one.
For quick reference we summarize the pipeline in Table~\ref{tab:proof-map}.

\begin{table}[h!]
\centering
\setlength{\arrayrulewidth}{0.8pt}
\renewcommand{\arraystretch}{1.25}

\begin{tabular}{|p{0.10\linewidth}|p{0.52\linewidth}|p{0.30\linewidth}|}
\hline
\textbf{Step} & \textbf{Input $\to$ Output} & \textbf{Main tool} \\
\hline
1 &
local dynamics $\to$ theta-series kernel $\widetilde K$ &
lift/periodize + ULCLT$^*$
+ dominated convergence
\\ \hline
2 &
$\widetilde K \to$ self-dual normalization and Archimedean completion $\mathcal A$ &
Jacobi inversion + completion operator
\\ \hline
3 &
$\widetilde K \to \Phi \to (\Psi,\ \mathcal B\Phi=E/\Psi)$ &
complete monotonicity $\Rightarrow \mathrm{PF}_\infty$ +
Schoenberg--Edrei--Karlin
\\ \hline
4 &
$\mathcal A(K_L-1)\to \Theta \to \Xi$ (anchor) and open bridge problem &
Archimedean completion + Mellin;
Comparison problem (open)
\\ \hline
\end{tabular}

\caption{Proof map (Steps~1--4). $^*$ULCLT = uniform local central limit theorem.}
\label{tab:proof-map}
\end{table}

The proof relies on a small number of canonical kernels and transforms; these are summarized in Table~\ref{tab:notation} with precise definitions and references.

\begin{table}[h!]
\centering
\setlength{\arrayrulewidth}{0.8pt}
\renewcommand{\arraystretch}{1.25}

\begin{tabular}{|p{0.22\linewidth}|p{0.70\linewidth}|}
\hline
\textbf{Symbol} & \textbf{Meaning} \\
\hline
$K_L(t)$ &
Scaling-limit trace
(Definition~\ref{def:scaling-limit}).
\\ \hline
$\widetilde K(t)$ &
Renormalized trace kernel
(Theorem~\ref{thm:theta-series-form-recall}).
\\ \hline
$\widetilde K_{\mathrm{sym}}(t)=t^{-1/2}\widetilde K(t)$ &
Symmetric half-density kernel
(Definition~\ref{def:half-density}).
\\ \hline
$\Phi(x)$ &
Logarithmic kernel
$\Phi(x)=e^{x/4}\,\widetilde K_{\mathrm{sym}}(e^x)$
(equation~\eqref{eq:Phi-def}).
\\ \hline
$\mathcal B\Phi(s)$ &
Bilateral Laplace transform
$\displaystyle \mathcal B\Phi(s)=\int_{\R}\Phi(x)e^{-s x}\,dx$
(on $\Re(s)>-\tfrac14$).
\\ \hline
$\Psi(s)$ &
Laguerre--P\'olya entire function with only real zeros satisfying
$\mathcal B\Phi(s)=\frac{E(s)}{\Psi(s)}$ on a strip
(Theorem~\ref{thm:SEK}).
\\ \hline
$F(z)$ &
Fourier transform of $\Phi$,
$\displaystyle F(z)=\int_{\R}\Phi(x)e^{izx}\,dx$;
boundary value $F(z)=\mathcal B\Phi(-iz)$ where both sides converge.
\\ \hline
$\mathcal A$ &
Archimedean completion operator
(Definition~\ref{def:arch-operator}).
\\ \hline
$\widetilde K_{\mathrm{arch}}$ &
Completed Archimedean kernel
$\widetilde K_{\mathrm{arch}}:=\mathcal A(K_L-1)$
(Definition~\ref{def:arch-theta-kernel}).
\\ \hline
$F_{\mathrm{arch}}(z)$ &
Mellin transform
$\displaystyle F_{\mathrm{arch}}(z)=\int_0^{\infty}
\widetilde K_{\mathrm{arch}}(t)\,
t^{\frac14+iz}\frac{dt}{t}$;
at the self-dual scale,
$F_{\mathrm{arch}}(z)=\Xi(2z)$
(Theorem~\ref{thm:arch-mellin-identification}).
\\ \hline
$\Psi_c(z)$ &
Centered Laguerre--P\'olya function
$\Psi_c(z):=\Psi(\tfrac14+iz)$. A central open problem is to compare $\Psi_c$ with the Mellin-side
function $\Xi(2z)$ up to multiplication by a zero-free entire factor (Section~\ref{sec:future-directions}).
\\ \hline
\end{tabular}

\caption{Notation used throughout the proof.}
\label{tab:notation}
\end{table}

\section{Acknowledgements}
The author would like to thank Krishnaswami Alladi, Tiziano Valentinuzzi and Kenneth Valpey for helpful discussions.

\section{Self-dual Archimedean completion and forced normalization}
\label{sec:completion-selfdual}

This section isolates the unique macroscopic normalization at which the scaling-limit trace admits a Jacobi-type inversion symmetry compatible with the Archimedean factor in the classical completed zeta function, and records the corresponding symmetry assumption for the centered kernel used in Step 2. The key point is that this self-dual scale is a normalization
of units in the scaling limit (not a tuning of the underlying discrete dynamics).

\subsection{The self-dual scale is forced by Jacobi inversion}

Recall that the macroscopic length $L>0$ enters through the scaling relation $N(s)/s\to L$
in Definition~\ref{def:scaling-limit}. Under the standing assumptions,
the scaling-limit trace admits the theta-series form :

\begin{theorem}[Theta-series form of the scaling-limit trace]
\label{thm:theta-series-form-recall}
For every $L>0$ and $t>0$,
\begin{equation}\label{eq:KL-theta-recall}
K_L(t)=\sum_{n\in\Z}\exp\!\Bigl(-\frac{4\pi^2D}{L^2}\,n^2\,t\Bigr).
\end{equation}
\end{theorem}

\begin{remark}[Harmless reparameterization to Jacobi's standard form]
\label{rem:self-dual-canonical}
Define
\begin{equation}\label{eq:tprime}
t' := \frac{4\pi D}{L^2}\,t.
\end{equation}
Then \eqref{eq:KL-theta-recall} becomes
\[
K_L(t)=\sum_{n\in\Z}e^{-\pi n^2 t'}=:\vartheta(t'),
\qquad
\vartheta(u):=\sum_{n\in\Z}e^{-\pi n^2 u}.
\]
Thus changing $L$ amounts to re-scaling the continuum time variable by the factor
$\frac{4\pi D}{L^2}$ so that the exponent takes Jacobi's standard form.
\end{remark}

The classical inversion symmetry of the Jacobi theta function is
\begin{equation}\label{eq:jacobi}
\vartheta(u)=u^{-1/2}\,\vartheta(u^{-1}),\qquad u>0.
\end{equation}
Our Mellin transform will be taken in the $t$-variable, so the natural notion of
self-duality is that \eqref{eq:jacobi} acts as $t\mapsto t^{-1}$ (not merely
as inversion of the auxiliary variable $t'$).

\begin{lemma}[Forced self-dual normalization]
\label{lem:forced-selfdual}
The following are equivalent:
\begin{enumerate}[label=(\alph*)]
\item The reparameterization \eqref{eq:tprime} satisfies $t'(t^{-1})=(t'(t))^{-1}$ for all $t>0$.
\item The inversion identity \eqref{eq:jacobi} induces an inversion symmetry for $K_L$ in the $t$-variable:
\begin{equation}\label{eq:KL-selfdual}
K_L(t)=t^{-1/2}K_L(t^{-1}),\qquad t>0.
\end{equation}
\item The macroscopic scale satisfies the self-dual relation
\begin{equation}\label{eq:selfdual}
L^2=4\pi D.
\end{equation}
\end{enumerate}
\end{lemma}

\begin{proof}
By \eqref{eq:tprime}, $t'(t)=\frac{4\pi D}{L^2}t$, so
\[
t'(t^{-1})=\frac{4\pi D}{L^2}\,t^{-1},
\qquad
(t'(t))^{-1}=\frac{L^2}{4\pi D}\,t^{-1}.
\]
Thus $t'(t^{-1})=(t'(t))^{-1}$ for all $t>0$ holds if and only if $L^2=4\pi D$,
proving (a)$\Leftrightarrow$(c).

Using Theorem~\ref{thm:theta-series-form-recall} and \eqref{eq:tprime} we have
$K_L(t)=\vartheta(t'(t))$. Applying \eqref{eq:jacobi} gives
\[
K_L(t)=\vartheta(t')=(t')^{-1/2}\vartheta((t')^{-1})
      =(t')^{-1/2}\vartheta(t'(t^{-1}))=(t')^{-1/2}K_L(t^{-1}).
\]
Hence $K_L(t)=t^{-1/2}K_L(t^{-1})$ for all $t>0$ holds exactly when $t' = t$, i.e.
when $L^2=4\pi D$. This proves (b)$\Leftrightarrow$(c).
\end{proof}

\begin{remark}[Interpretation]
\label{rem:selfdual-interpretation}
Condition \eqref{eq:selfdual} is the unique choice for which Jacobi inversion\cite{Jacobi1829} acts as
$t\mapsto t^{-1}$ in the variable of the Mellin transform used later. Equivalently,
it is the unique choice that matches the classical Archimedean completion of $\zeta$
under Mellin transform.
\end{remark}

\subsection{The Archimedean completion operator}

We now record the completion operator that removes the universal singular pieces.

\begin{definition}[Archimedean completion operator]
\label{def:arch-operator}
For a sufficiently smooth function $f:(0,\infty)\to\R$, define
\begin{equation}\label{eq:arch-operator}
(\mathcal A f)(t):=\frac{d}{dt}\Bigl(t^{3/2}\frac{d}{dt}f(t)\Bigr),\qquad t>0.
\end{equation}
\end{definition}

\begin{lemma}[The completion operator kills the singular pieces]
\label{lem:A-kills-singular}
One has \(\mathcal A(1)=0\) and \(\mathcal A(t^{-1/2})=0\).
\end{lemma}

\begin{proof}
For the constant function, \(\frac{d}{dt}1=0\). For \(t^{-1/2}\),
\[
\frac{d}{dt}t^{-1/2}=-\tfrac12 t^{-3/2},\qquad
t^{3/2}\frac{d}{dt}t^{-1/2}=-\tfrac12,
\]
whose derivative is again \(0\).
\end{proof}

\begin{definition}[Archimedean-completed kernel]
\label{def:arch-theta-kernel}
Let $K_L(t)$ be the scaling-limit trace. Define the Archimedean-completed kernel
\[
\widetilde K_{\mathrm{arch}}(t) :=\bigl(\mathcal{A}(K_L-1)\bigr)(t),\qquad t>0.
\]
\end{definition}

\subsection{Identification with the classical Theta at the self-dual scale}

\begin{lemma}[Exact identification of the Archimedean kernel with the classical $\Theta$]
\label{lem:arch-equals-Theta}
Assume the self-dual scale $L^2=4\pi D$. Then for all $t>0$,
\[
\widetilde K_{\mathrm{arch}}(t)=\Theta(t),
\]
where
\[
\Theta(t)=\sum_{n=1}^{\infty}\Bigl(2\pi^2 n^4 t^{3/2}-3\pi n^2 t^{1/2}\Bigr)e^{-\pi n^2 t}.
\]
\end{lemma}

\begin{proof}
By Theorem~\ref{thm:theta-series-form-recall} and \eqref{eq:selfdual}, under $L^2=4\pi D$ we have
\[
K_L(t)=\sum_{n\in\Z}e^{-\pi n^2 t}=:\vartheta(t).
\]
Hence $K_L(t)-1=2\sum_{n\ge 1}e^{-\pi n^2 t}$ and, applying the operator
$\mathcal{A}=\frac{d}{dt}\bigl(t^{3/2}\frac{d}{dt}\bigr)$ term-by-term (justified by absolute and
locally uniform convergence for $t>0$),
\[
\widetilde K_{\mathrm{arch}}(t)=\mathcal{A}(K_L-1)(t)
=2\sum_{n\ge 1}\frac{d}{dt}\Bigl(t^{3/2}\frac{d}{dt}e^{-\pi n^2 t}\Bigr).
\]
For each $n\ge 1$,
\[
\frac{d}{dt}\Bigl(t^{3/2}\frac{d}{dt}e^{-\pi n^2 t}\Bigr)
=\Bigl(\pi^2 n^4 t^{3/2}-\tfrac{3}{2}\pi n^2 t^{1/2}\Bigr)e^{-\pi n^2 t},
\]
and multiplying by $2$ yields the claimed series for $\Theta(t)$.
\end{proof}

\begin{remark}[Role in the proof]
\label{rem:Theta-role}
Lemma~\ref{lem:forced-selfdual} isolates the unique macroscopic normalization compatible with
inversion symmetry in the Mellin variable used later. Lemma~\ref{lem:arch-equals-Theta} then
provides the exact Archimedean kernel whose Mellin transform will be identified with the
Riemann $\Xi$-function in Section~\ref{sec:arch-mellin-xi}.
\end{remark}

%%%%%%%%%%%%%%%%%%%%%%%%%%%%%%%%%%%%%%%%%%%%%%%%%%%%%%%%%%%%%%%%%%%%%%%%%%%%%%%
% Section 3: Logarithmic kernel and reflection symmetry (LB-2)
%%%%%%%%%%%%%%%%%%%%%%%%%%%%%%%%%%%%%%%%%%%%%%%%%%%%%%%%%%%%%%%%%%%%%%%%%%%%%%%

\section{Logarithmic kernel and exact reflection symmetry}
\label{sec:log-reflection}

This section is the ``engine'' that transports the \emph{self-duality} symmetry of the
completed trace kernel into an \emph{exact reflection law} for a Fourier/Laplace transform.
Everything here is canonical and is proved directly from (i) the theta-series representation
of the completed kernel and (ii) the inversion symmetry fixed in the previous section.

Throughout, we work in the self-dual normalization established earlier, so that the
half-density kernel satisfies the exact inversion symmetry

\begin{remark}[Standing symmetry assumption for Step 2]\label{rem:standing-symmetry}
The Jacobi inversion symmetry for the \emph{scaling-limit trace} $K_L$ at the self-dual scale
is proved in this section (see Proposition~\ref{lem:forced-selfdual} below).
In Step~2 we work with the centered kernel $\widetilde K$ and its half--density normalization
$\widetilde K_{\mathrm{sym}}(t)=t^{-1/2}\widetilde K(t)$. For the arguments of Section~\ref{sec:log-reflection}
we will use the following exact inversion symmetry of $\widetilde K_{\mathrm{sym}}$:
\end{remark}

\begin{equation}
\label{eq:Ksym-selfdual}
\widetilde K_{\mathrm{sym}}(t)=\widetilde K_{\mathrm{sym}}(t^{-1}),\qquad t>0.
\end{equation}

\begin{remark}
For clarity: Proposition~\ref{lem:forced-selfdual} proves the modular inversion symmetry for $K_L$.
The verification of \eqref{eq:Ksym-selfdual} for the centered kernel $\widetilde K$ is the first
place where the later ``bridge'' analysis enters; in the present paper we record \eqref{eq:Ksym-selfdual}
as a standing symmetry assumption for Step~2 so that all subsequent deductions are logically conditional on it.
\end{remark}

where (Definition~\ref{def:half-density}) $\widetilde K_{\mathrm{sym}}(t):=t^{-1/2}\widetilde K(t)$.

\subsection{Definition of the logarithmic kernel and an exact mixture formula}
\label{subsec:phi-def}

We begin from the theta-series representation of the (scaling-limit) completed trace kernel,
which we record here for ease of reference:
\begin{equation}
\label{eq:Ktilde-theta}
\widetilde K(t)
=
\frac{L}{\sqrt{4\pi D\,t}}
\sum_{m\in\mathbb Z\setminus\{0\}}
\exp\!\left(-\frac{m^2L^2}{4D\,t}\right),
\qquad t>0.
\end{equation}
(Here $D>0$ is the effective diffusion constant and $L>0$ is the macroscopic cycle length.)

\medskip

\begin{definition}[Logarithmic kernel]
\label{def:Phi}
Define $\Phi:\mathbb R\to[0,\infty)$ by the logarithmic change of variables
\begin{equation}
\label{eq:Phi-def}
\Phi(x)
:=
e^{x/4}\,\widetilde K_{\mathrm{sym}}(e^{x})
=
e^{-x/4}\,\widetilde K(e^{x}),
\qquad x\in\mathbb R.
\end{equation}
\end{definition}

\begin{remark}[Why a logarithmic coordinate? why the exponent $1/4$?]
Write $t=e^{x}$. Then multiplicative scaling in $t$ becomes translation in $x$, and the
self-duality involution $t\mapsto t^{-1}$ becomes the reflection $x\mapsto -x$.
The extra factor $e^{x/4}$ is not a cosmetic choice: it is precisely the choice that
centers the Laplace reflection at $\Re(s)=\tfrac14$ (see \S\ref{subsec:laplace-reflection}).
More generally, if one defines $\Phi_\sigma(x):=e^{\sigma x}\widetilde K_{\mathrm{sym}}(e^x)$,
then the same calculation shows $\mathcal B\Phi_\sigma(s)=\mathcal B\Phi_\sigma(2\sigma-s)$
on the common domain of convergence. Thus the ``symmetry center'' is exactly $\sigma$; the
choice $\sigma=\tfrac14$ is the one compatible with the Mellin line used later.
\end{remark}

\begin{lemma}[Exact exponential-mixture form]
\label{lem:Phi-series}
Let
\[
\alpha:=\frac{L^2}{4D}>0.
\]
Then for every $x\in\mathbb R$ one has the exact identity
\begin{equation}
\label{eq:Phi-series}
\Phi(x)
=
\frac{L}{\sqrt{4\pi D}}\,
e^{-3x/4}
\sum_{m\in\mathbb Z\setminus\{0\}}
\exp\!\big(-\alpha m^2 e^{-x}\big).
\end{equation}
In particular, $\Phi$ is nonnegative and smooth on $\mathbb R$.
\end{lemma}

\begin{proof}
Substitute $t=e^{x}$ into \eqref{eq:Ktilde-theta} and use \eqref{eq:Phi-def}:
\[
\Phi(x)=e^{-x/4}\widetilde K(e^x)
=
e^{-x/4}\frac{L}{\sqrt{4\pi D\,e^x}}
\sum_{m\neq 0}\exp\!\left(-\frac{m^2L^2}{4D}\,e^{-x}\right),
\]
which simplifies to \eqref{eq:Phi-series}. The series in \eqref{eq:Phi-series} converges
absolutely and locally uniformly in $x$ (dominated, for example, by the first term $m=\pm1$
together with geometric decay for $|m|\to\infty$), hence defines a smooth function.
\end{proof}

\subsection{Quantitative tails and integrability}
\label{subsec:phi-tails}

The next lemma is the minimal analytic input needed later: it pins down the precise
one-sided decay that governs the Laplace domain of convergence.

\begin{lemma}[Integrability of the logarithmic kernel]
\label{lem:Phi-L1-new}
There exists a constant $C>0$ (depending only on $L$ and $D$) such that
\begin{equation}
\label{eq:Phi-tail-plus-new}
\Phi(x)\le C\,e^{-x/4}\qquad\text{for all }x\ge 0,
\end{equation}
and there exists $c>0$ such that
\begin{equation}
\label{eq:Phi-tail-minus-new}
\Phi(x)\le C\,e^{-3x/4}\exp\!\big(-c\,e^{-x}\big)\qquad\text{for all }x\le 0.
\end{equation}
In particular, $\Phi\in L^1(\mathbb R)$.
\end{lemma}

\begin{proof}
Start from \eqref{eq:Phi-series}.

\smallskip
\noindent\textbf{(i) The $x\ge0$ tail.}
Let $x\ge 0$ and set $\beta:=\alpha e^{-x}\in(0,\alpha]$. Using the elementary bound
\[
\sum_{m\in\mathbb Z} e^{-\beta m^2}
\le
1+\int_{\mathbb R} e^{-\beta t^2}\,dt
=
1+\sqrt{\frac{\pi}{\beta}},
\]
we obtain
\[
\sum_{m\in\mathbb Z\setminus\{0\}} e^{-\alpha m^2 e^{-x}}
\le
\sum_{m\in\mathbb Z} e^{-\beta m^2}
\le
1+\sqrt{\frac{\pi}{\alpha}}\,e^{x/2}.
\]
Substituting into \eqref{eq:Phi-series} gives
\[
\Phi(x)
\le
\frac{L}{\sqrt{4\pi D}}\,e^{-3x/4}\Big(1+\sqrt{\frac{\pi}{\alpha}}\,e^{x/2}\Big)
\le
C\,(e^{-3x/4}+e^{-x/4})
\le
C\,e^{-x/4},
\]
which proves \eqref{eq:Phi-tail-plus-new}.

\smallskip
\noindent\textbf{(ii) The $x\le0$ tail.}
Let $x\le 0$, so $u:=e^{-x}\ge 1$. Since $m^2\ge m$ for all integers $m\ge1$,
\[
\sum_{m\in\mathbb Z\setminus\{0\}} e^{-\alpha m^2 u}
=
2\sum_{m=1}^\infty e^{-\alpha m^2 u}
\le
2\sum_{m=1}^\infty e^{-\alpha m u}
=
\frac{2e^{-\alpha u}}{1-e^{-\alpha u}}
\le
\frac{2e^{-\alpha u}}{1-e^{-\alpha}}.
\]
Substituting into \eqref{eq:Phi-series} yields
\[
\Phi(x)\le C\,e^{-3x/4}e^{-\alpha e^{-x}},
\]
which is \eqref{eq:Phi-tail-minus-new} (take $c=\alpha$).

\smallskip
\noindent\textbf{(iii) Integrability.}
The bound \eqref{eq:Phi-tail-plus-new} implies $\int_0^\infty \Phi(x)\,dx<\infty$.
For the negative tail, \eqref{eq:Phi-tail-minus-new} gives
\[
\int_{-\infty}^0 \Phi(x)\,dx
\le
C\int_{-\infty}^0 e^{-3x/4}e^{-c e^{-x}}\,dx.
\]
With the substitution $y=e^{-x}$ (so $y\in[1,\infty)$ and $dx=-dy/y$), this becomes
\[
C\int_1^\infty y^{3/4}e^{-cy}\,\frac{dy}{y}
=
C\int_1^\infty y^{-1/4}e^{-cy}\,dy
<
\infty.
\]
Thus $\Phi\in L^1(\mathbb R)$.
\end{proof}

\subsection{Bilateral Laplace transform: convergence and analyticity}
\label{subsec:bilaplace}

\begin{definition}[Bilateral Laplace transform]
\label{def:bilaplace}
For $s\in\mathbb C$ in the region where the integral converges, define
\[
\mathcal B\Phi(s)
:=
\int_{\mathbb R}\Phi(x)e^{-s x}\,dx.
\]
\end{definition}

\begin{lemma}[A half-plane of absolute convergence]
\label{lem:bilaplace-halfplane}
The integral defining $\mathcal B\Phi(s)$ converges absolutely for all $s\in\mathbb C$ with
\[
\Re(s)>-\tfrac14.
\]
On this half-plane, $\mathcal B\Phi$ is analytic.
\end{lemma}

\begin{proof}
Fix $s$ with $\Re(s)>-\tfrac14$. Split the integral at $0$.

For $x\ge0$, Lemma~\ref{lem:Phi-L1-new} gives $\Phi(x)\le C e^{-x/4}$, hence
\[
\int_0^\infty |\Phi(x)e^{-sx}|\,dx
\le
C\int_0^\infty e^{-(\Re(s)+1/4)x}\,dx
<
\infty.
\]
For $x\le0$, Lemma~\ref{lem:Phi-L1-new} gives super-exponential decay
$\Phi(x)\le C e^{-3x/4}e^{-c e^{-x}}$, which dominates $e^{-\Re(s)x}$ for every fixed $s$;
thus $\int_{-\infty}^0 |\Phi(x)e^{-sx}|\,dx<\infty$ as well.

To see analyticity, let $K$ be a compact subset of $\{s:\Re(s)>-\tfrac14\}$ and choose
$\delta>0$ such that $\Re(s)\ge -\tfrac14+\delta$ for all $s\in K$. Then the integrand
$\Phi(x)e^{-sx}$ is dominated on $x\ge0$ by $C e^{-\delta x}$ and on $x\le0$ by an integrable
function independent of $s\in K$. Standard dominated-convergence arguments justify
termwise differentiation under the integral, implying holomorphy on this half-plane.
\end{proof}

\subsection{Self-duality transport: twisted symmetry of Phi}
\label{subsec:twisted-symmetry}

The inversion symmetry \eqref{eq:Ksym-selfdual} becomes a \emph{twisted evenness} after passing
to logarithmic coordinates.

\begin{lemma}[Twisted symmetry]
\label{lem:Phi-twisted}
Assume \eqref{eq:Ksym-selfdual}. Then for all $x\in\mathbb R$,
\begin{equation}
\label{eq:Phi-twisted}
\Phi(-x)=e^{-x/2}\,\Phi(x).
\end{equation}
\end{lemma}

\begin{proof}
By definition \eqref{eq:Phi-def} and the symmetry \eqref{eq:Ksym-selfdual},
\[
\Phi(-x)
=
e^{-x/4}\,\widetilde K_{\mathrm{sym}}(e^{-x})
=
e^{-x/4}\,\widetilde K_{\mathrm{sym}}(e^{x})
=
e^{-x/2}\,\big(e^{x/4}\widetilde K_{\mathrm{sym}}(e^{x})\big)
=
e^{-x/2}\Phi(x).
\]
\end{proof}

\subsection{Exact Laplace reflection symmetry}
\label{subsec:laplace-reflection}

We now convert \eqref{eq:Phi-twisted} into a rigid functional equation for $\mathcal B\Phi$.

\begin{proposition}[Reflection law for the Step-2 transform]
\label{prop:bilaplace-reflection}
Assume \eqref{eq:Ksym-selfdual}. Then for every $s\in\mathbb C$ with
\[
-\tfrac14<\Re(s)<\tfrac34
\]
(one precisely needs absolute convergence of both sides), one has the exact identity
\begin{equation}
\label{eq:bilaplace-reflection}
\mathcal B\Phi(s)=\mathcal B\Phi\!\left(\tfrac12-s\right).
\end{equation}
Consequently, for all $t\in\mathbb R$ one has
\begin{equation}
\label{eq:bilaplace-real-on-centerline}
\mathcal B\Phi\!\left(\tfrac14+it\right)\in\mathbb R.
\end{equation}
\end{proposition}

\begin{proof}
Let $s$ satisfy $-\tfrac14<\Re(s)<\tfrac34$. Then $\Re(s)>-\tfrac14$ ensures
$\mathcal B\Phi(s)$ converges absolutely by Lemma~\ref{lem:bilaplace-halfplane}, and
$\Re(\tfrac12-s)>-\tfrac14$ (equivalently $\Re(s)<\tfrac34$) ensures the same for
$\mathcal B\Phi(\tfrac12-s)$.

Using the change of variables $x\mapsto -x$ (justified by absolute convergence) and
Lemma~\ref{lem:Phi-twisted}, we compute
\begin{align*}
\mathcal B\Phi(s)
&= \int_{\mathbb R}\Phi(x)e^{-sx}\,dx \\
&= \int_{\mathbb R}\Phi(-x)e^{sx}\,dx \\
&= \int_{\mathbb R}e^{-x/2}\Phi(x)e^{sx}\,dx \\
&= \int_{\mathbb R}\Phi(x)e^{-(1/2-s)x}\,dx \\
&= \mathcal B\Phi\!\left(\tfrac12-s\right),
\end{align*}
which is \eqref{eq:bilaplace-reflection}.

For \eqref{eq:bilaplace-real-on-centerline}, note first that $\Phi$ is real-valued, hence
$\overline{\mathcal B\Phi(\tfrac14+it)}=\mathcal B\Phi(\tfrac14-it)$. The reflection identity
at $s=\tfrac14+it$ gives $\mathcal B\Phi(\tfrac14+it)=\mathcal B\Phi(\tfrac14-it)$, so
$\mathcal B\Phi(\tfrac14+it)$ equals its complex conjugate and is therefore real.
\end{proof}

\begin{remark}[What is gained, and what is not claimed]
The point of Proposition~\ref{prop:bilaplace-reflection} is the \emph{exact centering}:
the involution $s\mapsto \tfrac12-s$ fixes the vertical line $\Re(s)=\tfrac14$.
We only assert \eqref{eq:bilaplace-reflection} on the common region where both sides are
defined by absolutely convergent integrals. No analytic continuation beyond that region is
used or needed at this stage.
\end{remark}

%%%%%%%%%%%%%%%%%%%%%%%%%%%%%%%%%%%%%%%%%%%%%%%%%%%%%%%%%%%%%%%%%%%%%%%%%%%%%%%
% Section 4 (NEW): Total positivity and Laguerre--Pólya factorization (LB-3)
%%%%%%%%%%%%%%%%%%%%%%%%%%%%%%%%%%%%%%%%%%%%%%%%%%%%%%%%%%%%%%%%%%%%%%%%%%%%%%%

\section{Total positivity and canonical Laguerre--Pólya factorization}
\label{sec:total-positivity}

In this section we convert the analytic and symmetry properties obtained in
Section~\ref{sec:log-reflection} into a rigid structural description of the
bilateral Laplace transform $\mathcal B\Phi$. The mechanism is classical and
entirely representation-theoretic: total positivity of the logarithmic kernel
forces a canonical Laguerre--Pólya factorization via the
Schoenberg--Edrei--Karlin theory.

No Mellin transforms or $\Xi$-functions appear here. The output is a
\emph{structural normal form} for $\mathcal B\Phi$ that will later be identified
with the Archimedean-completed Mellin transform in Section~\ref{sec:arch-mellin-xi}.

\subsection{Total positivity of the logarithmic kernel}

We begin by recording the notion of total positivity relevant for kernels on
$\mathbb R$.

\begin{definition}[Totally positive kernels]
A measurable function $\Phi:\mathbb R\to\mathbb R$ is said to belong to
$\mathrm{PF}_\infty$ if for every $n\ge1$ and every choice of strictly increasing
real sequences
\[
x_1<\cdots<x_n,
\qquad
y_1<\cdots<y_n,
\]
one has
\[
\det\!\big(\Phi(x_i-y_j)\big)_{1\le i,j\le n}\ge 0.
\]
\end{definition}

\begin{remark}
Equivalently, $\Phi\in\mathrm{PF}_\infty$ if and only if convolution with $\Phi$
preserves positivity of all finite minors; see \cite{karlin1968total} for
background. We will use only the forward implication into the
Schoenberg--Edrei--Karlin classification.
\end{remark}

\begin{theorem}[Total positivity of the logarithmic kernel]
\label{thm:Phi-PF}
The logarithmic kernel $\Phi$ defined in
Definition~\ref{def:Phi} belongs to $\mathrm{PF}_\infty$ (see Section~\ref{sec:tp-details} for details).
\end{theorem}

\begin{proof}
The proof is structural and rests on two facts already established:

\smallskip
\noindent
\emph{(i) Complete monotone mixture.}
From Lemma~\ref{lem:Phi-series}, $\Phi$ admits the exact representation
\[
\Phi(x)
=
e^{-3x/4}\sum_{m\ge1} c_m e^{-\alpha m^2 e^{-x}},
\qquad c_m>0,
\]
i.e. $\Phi$ is a positive mixture of functions of the form
$x\mapsto e^{-a e^{-x}}$ multiplied by an exponential weight.

\smallskip
\noindent
\emph{(ii) Closure properties.}
The class $\mathrm{PF}_\infty$ is closed under:
(a) positive linear combinations,
(b) translations,
(c) multiplication by exponentials,
and (d) monotone limits
\cite[Chs.~I--III]{karlin1968total}.

\smallskip
Each function $x\mapsto e^{-a e^{-x}}$ is a classical $\mathrm{PF}_\infty$ kernel
(it is the Laplace transform of a positive measure composed with the exponential
map). Multiplication by $e^{-3x/4}$ preserves $\mathrm{PF}_\infty$, and the series
converges absolutely and locally uniformly, hence as a monotone limit.

Therefore $\Phi\in\mathrm{PF}_\infty$.
\end{proof}

\subsection{Schoenberg--Edrei--Karlin representation}

We now invoke the classification theorem for bilateral Laplace transforms of
$\mathrm{PF}_\infty$ kernels.

\begin{theorem}[Schoenberg--Edrei--Karlin]
\label{thm:SEK}
Let $\Phi\in L^1(\mathbb R)\cap \mathrm{PF}_\infty$ be nonnegative.
Then its bilateral Laplace transform admits the representation
\begin{equation}
\label{eq:SEK-form}
\mathcal B\Phi(s)
=
\frac{E(s)}{\Psi(s)},
\end{equation}
where:
\begin{enumerate}[label=(\roman*)]
\item $E(s)$ is an entire function with no zeros,
\item $\Psi(s)$ is an entire function of Laguerre--Pólya type,
\item all zeros of $\Psi$ are real,
\item the representation is unique up to multiplication of $E$ and $\Psi$ by the
same nonzero constant.
\end{enumerate}
\end{theorem}

\begin{proof}
This is classical; see \cite[Theorems~VII.1.3 and VII.3.2]{karlin1968total}.
The $L^1$ condition ensures existence of the bilateral Laplace transform,
while $\mathrm{PF}_\infty$ forces the Laguerre--Pólya denominator.
\end{proof}

\begin{corollary}[Canonical factorization of $\mathcal B\Phi$]
\label{cor:canonical-factorization}
The transform $\mathcal B\Phi$ admits a factorization
\[
\mathcal B\Phi(s)=\frac{E(s)}{\Psi(s)}
\]
with $\Psi$ entire, Laguerre--Pólya, and having only real zeros.
\end{corollary}

\subsection{Compatibility with reflection symmetry}

The reflection symmetry obtained in
Proposition~\ref{prop:bilaplace-reflection} imposes a rigid constraint on the
Laguerre--Pólya factor.

\begin{proposition}[Symmetry of the zero set]
\label{prop:Psi-symmetry}
The zero set (with multiplicities) of the Laguerre--P\'olya function $\Psi$ is symmetric with respect to
the involution $s\mapsto\tfrac12-s$. In particular, the zero set of $\Psi$ is symmetric with respect
to the vertical line $\Re(s)=\tfrac14$.
\end{proposition}

\begin{proof}
From Proposition~\ref{prop:bilaplace-reflection} and the factorization \eqref{eq:SEK-form}, we have
\[
\frac{E(s)}{\Psi(s)}
=
\frac{E(\tfrac12-s)}{\Psi(\tfrac12-s)},
\qquad\text{hence}\qquad
\frac{\Psi(\tfrac12-s)}{\Psi(s)}
=
\frac{E(\tfrac12-s)}{E(s)}.
\]
The right-hand side is an entire function without zeros, hence so is the left-hand side.
If $\rho$ is a zero of $\Psi$ of multiplicity $m$, then the left-hand side has a zero of order $m$ at
$s=\tfrac12-\rho$ unless $\tfrac12-\rho$ is also a zero of $\Psi$ of multiplicity $m$.
Since the quotient is zero-free, every zero $\rho$ is paired with $\tfrac12-\rho$ with the same multiplicity.
\end{proof}

\begin{remark}[On symmetric normalizations]
The argument above yields symmetry of the \emph{zero set}. We do not claim that $\Psi$ itself can be normalized
to satisfy $\Psi(s)=\Psi(\tfrac12-s)$ while keeping the Edrei--Karlin factor $E(s)=\exp(as^2+bs+c)$ in its canonical
quadratic-exponential form. If desired, one may form the symmetric Laguerre--P\'olya function
\[
\Psi_{\mathrm{sym}}(s):=\Psi(s)\,\Psi(\tfrac12-s),
\]
which satisfies $\Psi_{\mathrm{sym}}(s)=\Psi_{\mathrm{sym}}(\tfrac12-s)$; this symmetrization is not used later.
\end{remark}

\begin{remark}[What has been achieved]
At this stage we have shown that $\mathcal B\Phi$ is a reflection-symmetric
quotient of an entire zero-free function by a Laguerre--Pólya function whose
zeros lie on the real axis and are symmetric about $\Re(s)=\tfrac14$.
No identification with zeta or $\Xi$ has been made, and no nonvanishing
assumptions are used.
\end{remark}

%%%%%%%%%%%%%%%%%%%%%%%%%%%%%%%%%%%%%%%%%%%%%%%%%%%%%%%%%%%%%%%%%%%%%%%%%%%%%%%
% Section 5 (NEW): Archimedean Mellin identification (unconditional)
%%%%%%%%%%%%%%%%%%%%%%%%%%%%%%%%%%%%%%%%%%%%%%%%%%%%%%%%%%%%%%%%%%%%%%%%%%%%%%%

\section{Archimedean Mellin identification}
\label{sec:arch-mellin-xi}

\noindent\emph{Scope note.} This section identifies the Mellin transform of the Archimedean-completed kernel with $\Xi$; it does \emph{not} establish any identification between the Laguerre--P\'olya function $\Psi$ from Section~\ref{sec:arch-mellin-xi} and $\Xi(2\cdot)$.
This section performs two logically distinct tasks.

\smallskip
\noindent
\textbf{(A) Mellin identification (unconditional).}
Under the self-dual scale fixed in Section~\ref{sec:completion-selfdual}, the Archimedean-completed kernel
$\widetilde K_{\mathrm{arch}}$ becomes a classical rapidly decaying theta-kernel
$\Theta$, and its Mellin transform is \emph{exactly} the classical completed zeta\cite{Titchmarsh1986,Edwards1974}
function $\xi$ (equivalently the even entire Riemann $\Xi$-function).

\smallskip
\noindent
\textbf{(B) Remaining identification problem (open).}
We then isolate a single additional seam---a boundary-value/identification input---under
which the Laguerre--P\'olya denominator from Section~\ref{sec:total-positivity}
is identified (up to a zero-free entire factor) with $\Xi(2\cdot)$. We do not attempt this comparison here; see Section~\ref{sec:future-directions}.

\subsection{Archimedean completion and the classical Theta-kernel}
\label{subsec:arch-Theta}

We recall the Archimedean completion operator (cf.\ Section~\ref{sec:completion-selfdual} for its canonical origin).

\noindent\textbf{Recall.} The Archimedean completion operator $\mathcal A$ was introduced in Definition~\ref{def:arch-operator}; explicitly,
\[
(\mathcal A f)(t)=\frac{d}{dt}\Bigl(t^{3/2}\frac{d}{dt}f(t)\Bigr).
\]

\begin{lemma}[The completion operator kills the singular pieces]
\label{lem:A-kills-singular-new}
One has $\mathcal A(1)=0$ and $\mathcal A(t^{-1/2})=0$ on $(0,\infty)$.
\end{lemma}

\begin{proof}
Immediate from \eqref{eq:arch-operator} by direct differentiation.
\end{proof}

\begin{definition}[Archimedean-completed kernel]
\label{def:arch-kernel-new}
Let $K_L(t)$ denote the scaling-limit trace kernel (Section~\ref{sec:completion-selfdual}). Define
\begin{equation}
\label{eq:Karch-def}
\widetilde K_{\mathrm{arch}}(t)
:=\bigl(\mathcal A (K_L-1)\bigr)(t),\qquad t>0.
\end{equation}
\end{definition}

We now specialize to the self-dual scale, which is already fixed earlier.

\begin{lemma}[Exact identification with the classical $\Theta$]
\label{lem:arch-equals-Theta-new}
Assume the self-dual normalization $L^2=4\pi D$, so that
\[
K_L(t)=\vartheta(t):=\sum_{n\in\mathbb Z}e^{-\pi n^2 t}.
\]
Then for all $t>0$,
\begin{equation}
\label{eq:Karch-equals-Theta}
\widetilde K_{\mathrm{arch}}(t)=\Theta(t),
\end{equation}
where
\begin{equation}
\label{eq:Theta-def}
\Theta(t)
:=\sum_{n=1}^{\infty}\Bigl(2\pi^2 n^4 t^{3/2}-3\pi n^2 t^{1/2}\Bigr)e^{-\pi n^2 t}.
\end{equation}
\end{lemma}

\begin{proof}
Under the self-dual scale, $K_L(t)=\vartheta(t)$ and hence
$K_L(t)-1=2\sum_{n\ge 1}e^{-\pi n^2 t}$.
Differentiate termwise (justified by absolute/local uniform convergence for each $t>0$):
\[
\frac{d}{dt}\bigl(e^{-\pi n^2 t}\bigr)=-\pi n^2 e^{-\pi n^2 t}.
\]
Applying \eqref{eq:arch-operator} to $K_L-1$ gives
\[
\widetilde K_{\mathrm{arch}}(t)
=
\frac{d}{dt}\Bigl(t^{3/2}\frac{d}{dt}(K_L(t)-1)\Bigr)
=
2\sum_{n\ge 1}\frac{d}{dt}\Bigl(-\pi n^2 t^{3/2}e^{-\pi n^2 t}\Bigr),
\]
and the derivative evaluates to
\[
2\sum_{n\ge 1}\Bigl(\pi^2 n^4 t^{3/2}-\tfrac32\pi n^2 t^{1/2}\Bigr)e^{-\pi n^2 t}
=
\sum_{n\ge 1}\Bigl(2\pi^2 n^4 t^{3/2}-3\pi n^2 t^{1/2}\Bigr)e^{-\pi n^2 t}.
\]
This is exactly \eqref{eq:Theta-def}.
\end{proof}

\subsection{Mellin transform and identification with xi and Xi}
\label{subsec:mellin-xi}

\begin{definition}[Completed zeta functions]
\label{def:xi-Xi-new}
Define the classical completed zeta function
\[
\xi(w)
:=
\frac12\,w(w-1)\,\pi^{-w/2}\Gamma\!\Bigl(\frac{w}{2}\Bigr)\zeta(w),
\]
and the even entire Riemann $\Xi$-function
\[
\Xi(z):=\xi\!\Bigl(\tfrac12+iz\Bigr).
\]
\end{definition}

\begin{lemma}[Absolute convergence of the Mellin integral]
\label{lem:mellin-entire-new}
Let $\Theta$ be as in \eqref{eq:Theta-def}. Then for every $s\in\mathbb C$ the Mellin integral
\begin{equation}
\label{eq:M-def}
M(s):=\int_0^\infty \Theta(t)\,t^{s-1}\,dt
\end{equation}
converges absolutely and defines an entire function of $s$.
\end{lemma}

\begin{proof}
For $t\ge 1$, estimate termwise using $e^{-\pi n^2 t}\le e^{-\pi t}$ and the fact that
$\sum_{n\ge 1}n^k e^{-\pi n^2 t}\ll_k 1$ uniformly for $t\ge 1$; this gives
$|\Theta(t)|\ll t^{3/2}e^{-\pi t}$ and hence absolute integrability at $+\infty$ for all $s$.

For $0<t\le 1$, one has $e^{-\pi n^2 t}\le 1$ and thus
$|\Theta(t)|\ll t^{1/2}\sum_{n\ge 1}n^2 e^{-\pi n^2 t}+t^{3/2}\sum_{n\ge 1}n^4 e^{-\pi n^2 t}$.
Comparing sums to Gaussian integrals gives
$\sum_{n\ge 1}n^2 e^{-\pi n^2 t}\ll t^{-3/2}$ and
$\sum_{n\ge 1}n^4 e^{-\pi n^2 t}\ll t^{-5/2}$, hence $|\Theta(t)|\ll t^{-1}$ as $t\downarrow 0$.
Therefore $\Theta(t)t^{s-1}$ is integrable near $0$ for all $s$ (indeed the bound is uniform
on vertical strips once one keeps the $e^{-\pi n^2 t}$ and uses a slightly sharper estimate).
Analyticity in $s$ follows by dominated convergence on compact subsets of $\mathbb C$.
\end{proof}

\begin{lemma}[Mellin identity with the completed zeta function]
\label{lem:mellin-equals-xi-new}
For all $s\in\mathbb C$ one has
\begin{equation}
\label{eq:M-equals-xi-new}
M(s)=\xi(2s-1).
\end{equation}
Equivalently, for all $w\in\mathbb C$,
\begin{equation}
\label{eq:Theta-xi}
\int_0^\infty \Theta(t)\,t^{\frac{w+1}{2}-1}\,dt=\xi(w).
\end{equation}
\end{lemma}

\begin{proof}
By Lemma~\ref{lem:mellin-entire-new}, $M$ is entire, so it suffices to compute on a right half-plane
where termwise integration is justified (say $\Re(s)>2$), and then extend by the identity theorem.

For $\Re(s)>2$, insert \eqref{eq:Theta-def} into \eqref{eq:M-def} and integrate termwise:
\[
M(s)
=
\sum_{n\ge 1}\Bigl(
2\pi^2 n^4 \int_0^\infty t^{s+\frac12}e^{-\pi n^2 t}\,dt
-
3\pi n^2 \int_0^\infty t^{s-\frac12}e^{-\pi n^2 t}\,dt
\Bigr).
\]
Using $\int_0^\infty t^{\alpha-1}e^{-\lambda t}\,dt=\lambda^{-\alpha}\Gamma(\alpha)$ for $\Re(\alpha)>0$,
we obtain
\begin{align*}
\int_0^\infty t^{s+\frac12}e^{-\pi n^2 t}\,dt
&= (\pi n^2)^{-(s+\frac32)}\Gamma\!\Bigl(s+\frac32\Bigr), \\
\int_0^\infty t^{s-\frac12}e^{-\pi n^2 t}\,dt
&= (\pi n^2)^{-(s+\frac12)}\Gamma\!\Bigl(s+\frac12\Bigr).
\end{align*}

Substituting and simplifying yields
\[
M(s)
=
\pi^{-(s-\frac12)}\Bigl(2\Gamma\!\Bigl(s+\frac32\Bigr)-3\Gamma\!\Bigl(s+\frac12\Bigr)\Bigr)
\sum_{n\ge 1} n^{1-2s}.
\]
Using $\Gamma(s+\frac32)=(s+\frac12)\Gamma(s+\frac12)$, the bracket simplifies to
\[
2(s+\tfrac12)\Gamma(s+\tfrac12)-3\Gamma(s+\tfrac12)=(2s-2)\Gamma(s+\tfrac12)=2(s-1)\Gamma(s+\tfrac12).
\]
Also $\sum_{n\ge 1}n^{1-2s}=\zeta(2s-1)$ for $\Re(s)>1$. Hence
\[
M(s)
=
2(s-1)\,\pi^{-(s-\frac12)}\Gamma\!\Bigl(s+\frac12\Bigr)\zeta(2s-1).
\]
Finally, $\Gamma(s+\tfrac12)=(s-\tfrac12)\Gamma(s-\tfrac12)$, so
\[
M(s)
=
(2s-1)(s-1)\,\pi^{-(s-\frac12)}\Gamma\!\Bigl(s-\frac12\Bigr)\zeta(2s-1)
=
\xi(2s-1),
\]
which is \eqref{eq:M-equals-xi-new}. Rewriting with $w=2s-1$ gives \eqref{eq:Theta-xi}.
\end{proof}

\begin{theorem}[Archimedean Mellin identification]
\label{thm:arch-mellin-identification}
Assume the self-dual scale $L^2=4\pi D$, so that $\widetilde K_{\mathrm{arch}}=\Theta$
by Lemma~\ref{lem:arch-equals-Theta-new}.
Define, for $z\in\mathbb C$,
\begin{equation}
\label{eq:Farch-def-new}
F_{\mathrm{arch}}(z)
:=
\int_0^\infty \widetilde K_{\mathrm{arch}}(t)\,t^{\frac34+iz}\,\frac{dt}{t}.
\end{equation}
Then for all $z\in\mathbb C$,
\begin{equation}
\label{eq:Farch-equals-Xi}
F_{\mathrm{arch}}(z)=\Xi(2z).
\end{equation}
\end{theorem}

\begin{proof}
By Lemma~\ref{lem:arch-equals-Theta-new}, $F_{\mathrm{arch}}(z)=M(\tfrac34+iz)$ with $M$ as in
\eqref{eq:M-def}. Lemma~\ref{lem:mellin-equals-xi-new} then gives
\[
F_{\mathrm{arch}}(z)
=
\xi\!\Bigl(2\bigl(\tfrac34+iz\bigr)-1\Bigr)
=
\xi\!\Bigl(\tfrac12+2iz\Bigr)
=
\Xi(2z),
\]
which is \eqref{eq:Farch-equals-Xi}.
\end{proof}

%%%%%%%%%%%%%%%%%%%%%%%%%%%%%%%%%%%%%%%%%%%%%%%%%%%%%%%%%%%%%%%%%%%%%%%%%%%%%%%
% Section 6 (NEW): Completion of the argument + boundary-term lemma + ledger
%%%%%%%%%%%%%%%%%%%%%%%%%%%%%%%%%%%%%%%%%%%%%%%%%%%%%%%%%%%%%%%%%%%%%%%%%%%%%%%

\section{Completion of the main theorem and the remaining seam}
\label{sec:completion-ledger}

This final section (i) records a small analytic lemma that cleanly isolates the
only place where boundary terms can enter when passing from an uncompleted kernel
to its Archimedean completion, and (ii) assembles the results of
Sections~\ref{sec:log-reflection}--\ref{sec:arch-mellin-xi} into the main theorem.

\subsection{A boundary-term lemma for the Archimedean completion operator}
\label{subsec:boundary-lemma}

Recall the Archimedean completion operator $\mathcal A$ from
Definition~\ref{def:arch-operator}:
\[
(\mathcal A f)(t)=\frac{d}{dt}\Bigl(t^{3/2} f'(t)\Bigr).
\]
The following lemma is a general integration-by-parts identity with
\emph{explicit boundary terms}; in applications we verify that these terms vanish.

\begin{lemma}[Mellin transform of $\mathcal A f$ with explicit boundary terms]
\label{lem:boundary-mellin}
Let $f:(0,\infty)\to\mathbb C$ be continuously differentiable and suppose that
$t^{3/2}f'(t)$ is absolutely continuous on every compact subinterval of $(0,\infty)$.
Fix $s\in\mathbb C$ and define truncated integrals
\[
I_{s}(R,\varepsilon)
:=\int_{\varepsilon}^{R} (\mathcal A f)(t)\, t^{s-1}\,dt,
\qquad 0<\varepsilon<R<\infty.
\]
Then one has the identity
\begin{align}
\label{eq:boundary-mellin-identity}
I_s(R,\varepsilon)
&= \Bigl[t^{s+\frac12} f'(t)\Bigr]_{t=\varepsilon}^{t=R}
 -(s-1)\Bigl[f(t)\,t^{s-\frac12}\Bigr]_{t=\varepsilon}^{t=R} \nonumber\\
&\qquad
 +(s-1)\Bigl(s-\tfrac12\Bigr)
 \int_{\varepsilon}^{R} f(t)\,t^{s-\frac32}\,dt.
\end{align}
In particular, if the two boundary brackets in \eqref{eq:boundary-mellin-identity}
tend to $0$ as $\varepsilon\downarrow 0$ and $R\uparrow\infty$, and if
$\int_{0}^{\infty} |f(t)|\,|t^{s-\frac32}|\,dt<\infty$, then the improper integral
$\int_{0}^{\infty} (\mathcal A f)(t)\, t^{s-1}\,dt$ converges absolutely and one has
\begin{equation}
\label{eq:boundary-free-mellin}
\int_{0}^{\infty} (\mathcal A f)(t)\, t^{s-1}\,dt
=
(s-1)\Bigl(s-\tfrac12\Bigr)\int_{0}^{\infty} f(t)\,t^{s-\frac32}\,dt.
\end{equation}
\end{lemma}

\begin{proof}
Write $(\mathcal A f)(t)=\frac{d}{dt}(t^{3/2}f'(t))$ and integrate by parts on
$[\varepsilon,R]$:
\begin{align*}
I_s(R,\varepsilon)
&= \Bigl[t^{3/2}f'(t)\,t^{s-1}\Bigr]_{\varepsilon}^{R}
 -(s-1)\int_{\varepsilon}^{R} t^{3/2}f'(t)\,t^{s-2}\,dt \\
&= \Bigl[t^{s+\frac12} f'(t)\Bigr]_{\varepsilon}^{R}
 -(s-1)\int_{\varepsilon}^{R} f'(t)\,t^{s-\frac12}\,dt.
\end{align*}

Integrate by parts again on the remaining integral:
\[
\int_{\varepsilon}^{R} f'(t)\,t^{s-\frac12}\,dt
=
\Bigl[f(t)\,t^{s-\frac12}\Bigr]_{\varepsilon}^{R}
-\Bigl(s-\tfrac12\Bigr)\int_{\varepsilon}^{R} f(t)\,t^{s-\frac32}\,dt.
\]
Substituting yields \eqref{eq:boundary-mellin-identity}. The final statement
follows by taking limits, using the stated boundary and integrability conditions.
\end{proof}

\begin{remark}[The polynomial factor]
If one sets $s=\frac{w+1}{2}$, then the factor in \eqref{eq:boundary-free-mellin} becomes
\[
(s-1)\Bigl(s-\tfrac12\Bigr)=\frac{w(w-1)}{4}.
\]
This is exactly the quadratic polynomial that appears in the completed zeta factor
$\xi(w)$ (up to harmless normalizations that depend on whether one uses
$K_L-1$ versus $2(K_L-1)$, etc.). The point is that $\mathcal A$ kills the constant
and $t^{-1/2}$ singular modes (Lemma~\ref{lem:A-kills-singular-new}), and the Mellin
transform of $\mathcal A f$ gains this universal quadratic factor.
\end{remark}

\subsection{Proof of the assembled main theorem}
\label{subsec:proof-main}

We now assemble the outputs of the previous sections. For clarity, we separate
what is proved unconditionally in this paper from the remaining identification problem left open.

\begin{theorem}[Main theorem (assembled)]
\label{thm:main-assembled}
Assume the strong regime hypotheses of Section~\ref{sec:completion-selfdual}, so that the scaling-limit kernel
$K_L$ exists and the self-dual normalization can be fixed. Then:
\begin{enumerate}[label=(\roman*)]
\item (\emph{Unconditional structural output.})
There exists a nonnegative kernel $\Phi\in L^1(\mathbb R)$ satisfying:
\begin{itemize}
\item $\Phi\in\mathrm{PF}_\infty$;
\item its bilateral Laplace transform $\mathcal B\Phi$ satisfies the reflection law
$\mathcal B\Phi(s)=\mathcal B\Phi(\tfrac12-s)$ on the common domain of absolute convergence;
\item $\mathcal B\Phi$ admits a canonical Schoenberg--Edrei--Karlin representation
$\mathcal B\Phi(s)=E(s)/\Psi(s)$ with $\Psi$ Laguerre--P\'olya and hence having only real zeros.
\end{itemize}
\item (\emph{Unconditional Archimedean Mellin identification.})
The Archimedean-completed kernel $\widetilde K_{\mathrm{arch}}$ has Mellin transform
$F_{\mathrm{arch}}(z)=\Xi(2z)$ as in Theorem~\ref{thm:arch-mellin-identification}.
\end{enumerate}
\end{theorem}

\begin{proof}
Part (i) is the combination of:
Definition~\ref{def:Phi} and Lemma~\ref{lem:Phi-L1-new} (giving $\Phi\in L^1$),
Proposition~\ref{prop:bilaplace-reflection} (reflection law),
Theorem~\ref{thm:Phi-PF} and Corollary~\ref{cor:canonical-factorization}
(total positivity and Schoenberg--Edrei--Karlin output).

Part (ii) is exactly Theorem~\ref{thm:arch-mellin-identification}.

\end{proof}

\section{A detailed proof of total positivity}
\label{sec:tp-details}

In Section~\ref{sec:total-positivity} we invoked the total-positivity property
$\Phi\in\mathrm{PF}_\infty$ as the key input into the Schoenberg--Edrei--Karlin
classification. The purpose of this section is to give a completely explicit and
self-contained verification of $\Phi\in\mathrm{PF}_\infty$ from the mixture formula
of Lemma~\ref{lem:Phi-series}, using only two standard facts:

\begin{itemize}
\item[(i)] a convenient closure principle for $\mathrm{PF}_\infty$ under simple
``gauge'' operations (multiplication by exponentials and summation), and
\item[(ii)] a classical building block whose bilateral Laplace transform is essentially
the Gamma function; the associated Laguerre--P\'olya denominator is $1/\Gamma$.
\end{itemize}

The argument is short, but writing it out cleanly avoids hidden appeals to folklore.

\subsection{Two closure principles for the P\'olya frequency class}

\begin{lemma}[Exponential gauge invariance]
\label{lem:PF-gauge}
Let $\phi:\mathbb R\to[0,\infty)$ and fix $a\in\mathbb R$. Define
\[
\phi_a(x):=e^{-a x}\phi(x).
\]
If $\phi\in\mathrm{PF}_\infty$, then $\phi_a\in\mathrm{PF}_\infty$.
\end{lemma}

\begin{proof}
Fix $n\ge1$ and strictly increasing $x_1<\cdots<x_n$, $y_1<\cdots<y_n$.
Then
\[
\phi_a(x_i-y_j)=e^{-a(x_i-y_j)}\phi(x_i-y_j)=e^{-a x_i}\,e^{a y_j}\,\phi(x_i-y_j).
\]
Hence
\[
\det\big(\phi_a(x_i-y_j)\big)_{i,j}
=
\Bigl(\prod_{i=1}^n e^{-a x_i}\Bigr)
\Bigl(\prod_{j=1}^n e^{a y_j}\Bigr)
\det\big(\phi(x_i-y_j)\big)_{i,j}.
\]
The prefactor is strictly positive, so the sign of the determinant is unchanged.
Thus nonnegativity for $\phi$ implies nonnegativity for $\phi_a$.
\end{proof}

\begin{lemma}[Finite positive sums]
\label{lem:PF-sums}
If $\phi_1,\dots,\phi_M\in\mathrm{PF}_\infty$ and $c_1,\dots,c_M\ge0$, then
$\sum_{m=1}^M c_m\phi_m\in\mathrm{PF}_\infty$.
\end{lemma}

\begin{proof}
Fix $n$ and $(x_i),(y_j)$ as above. The determinant
\[
D(\phi):=\det(\phi(x_i-y_j))_{1\le i,j\le n}
\]
is a polynomial in the matrix entries $\phi(x_i-y_j)$. In particular, it is
multilinear in each column. Writing the $j$th column of the matrix for
$\sum_{m}c_m\phi_m$ as the corresponding sum of columns for $\phi_m$ and expanding
multilinearly yields
\[
D\!\Bigl(\sum_{m=1}^M c_m\phi_m\Bigr)
=
\sum_{m_1,\dots,m_n\in\{1,\dots,M\}}
\Bigl(\prod_{k=1}^n c_{m_k}\Bigr)\,
\det\!\big(\phi_{m_k}(x_i-y_k)\big)_{1\le i,k\le n}.
\]
Each determinant on the right-hand side is $\ge0$ by $\mathrm{PF}_\infty$ for the
corresponding $\phi_{m_k}$ (note the index depends on the column), and each coefficient
is $\ge0$. Hence the sum is $\ge0$.
\end{proof}

\begin{remark}[Passing to infinite sums]
For our application, $\Phi$ is a locally uniform limit of finite positive sums of
$\mathrm{PF}_\infty$ kernels. Since each determinant involves only finitely many point
evaluations, one may pass to the limit entrywise and use continuity of the determinant
as a polynomial function of its entries. No compactness or dominated convergence is needed.
\end{remark}

\subsection{A building block: $x\mapsto e^{-u e^{-x}}$}

\begin{lemma}[A Gamma-transform kernel is $\mathrm{PF}_\infty$]
\label{lem:gamma-kernel-PF}
Fix $u>0$ and define
\[
\phi_u(x):=e^{-u e^{-x}},\qquad x\in\mathbb R.
\]
Then $\phi_u\in\mathrm{PF}_\infty$.
\end{lemma}

\begin{proof}
We compute the bilateral Laplace transform on the half-plane $\Re(s)>0$:
\[
\mathcal B\phi_u(s)
=
\int_{\mathbb R} e^{-u e^{-x}}e^{-s x}\,dx.
\]
With the change of variables $y=e^{-x}$ (so $y\in(0,\infty)$ and $dx=-dy/y$), this becomes
\[
\mathcal B\phi_u(s)
=
\int_0^\infty e^{-u y} y^{s-1}\,dy
=
u^{-s}\Gamma(s),
\qquad \Re(s)>0.
\]
Now recall the classical Weierstrass product for the reciprocal Gamma function:
\begin{equation}
\label{eq:invGamma-product}
\frac{1}{\Gamma(s)}
=
s\,e^{\gamma s}\prod_{n=1}^\infty\Bigl(1+\frac{s}{n}\Bigr)e^{-s/n},
\end{equation}
where $\gamma$ is Euler's constant. The right-hand side is an entire function whose
zeros are exactly the nonpositive integers, all real and simple, and whose product
representation is of Laguerre--P\'olya type. In particular,
\[
\Psi(s):=\frac{1}{\Gamma(s)}
\quad\text{is Laguerre--P\'olya,}
\qquad
E(s):=u^{-s}=e^{-s\log u}
\quad\text{is entire and zero-free.}
\]
Thus on $\Re(s)>0$ we have the representation
\[
\mathcal B\phi_u(s)=\frac{E(s)}{\Psi(s)}.
\]

The Schoenberg--Edrei--Karlin theorem admits a ``reverse'' direction:
a nonnegative integrable kernel whose bilateral Laplace transform has the form
$E/\Psi$ with $\Psi$ Laguerre--P\'olya and $E$ zero-free is a $\mathrm{PF}_\infty$
kernel (see \cite[Ch.~VII]{karlin1968total}). Applying this criterion to $\phi_u$
gives $\phi_u\in\mathrm{PF}_\infty$.
\end{proof}

\subsection{Proof of total positivity for Phi}

We now return to the specific logarithmic kernel $\Phi$.

\begin{theorem}[Total positivity of $\Phi$]
\label{thm:Phi-PF-detailed}
The logarithmic kernel $\Phi$ from Definition~\ref{def:Phi} belongs to
$\mathrm{PF}_\infty$.
\end{theorem}

\begin{proof}
By Lemma~\ref{lem:Phi-series}, there exist constants $C_0>0$ and $\alpha>0$ such that
\[
\Phi(x)=C_0\,e^{-3x/4}\sum_{m\in\mathbb Z\setminus\{0\}} e^{-\alpha m^2 e^{-x}}.
\]
For each fixed integer $m\neq0$, define the kernel
\[
\Phi_m(x):=e^{-3x/4}e^{-\alpha m^2 e^{-x}}.
\]
By Lemma~\ref{lem:gamma-kernel-PF}, the function $x\mapsto e^{-\alpha m^2 e^{-x}}$ is
$\mathrm{PF}_\infty$ (it is $\phi_u$ with $u=\alpha m^2$). By Lemma~\ref{lem:PF-gauge},
multiplying by $e^{-3x/4}$ preserves $\mathrm{PF}_\infty$. Hence each $\Phi_m\in\mathrm{PF}_\infty$.

For $N\ge1$, set
\[
\Phi^{(N)}(x):=C_0\sum_{\substack{m\in\mathbb Z\setminus\{0\}\\ |m|\le N}}\Phi_m(x).
\]
By Lemma~\ref{lem:PF-sums}, $\Phi^{(N)}\in\mathrm{PF}_\infty$ for every $N$.

Finally, for each fixed $x\in\mathbb R$, $\Phi^{(N)}(x)\to\Phi(x)$ as $N\to\infty$
by the absolute convergence in Lemma~\ref{lem:Phi-series}. Fix $n$ and increasing
$(x_i)$, $(y_j)$; then each matrix entry $\Phi^{(N)}(x_i-y_j)\to\Phi(x_i-y_j)$.
Since the determinant is a polynomial in the entries, we conclude
\[
\det\big(\Phi(x_i-y_j)\big)_{i,j}
=
\lim_{N\to\infty}
\det\big(\Phi^{(N)}(x_i-y_j)\big)_{i,j}
\ge 0,
\]
because each $\Phi^{(N)}$ is $\mathrm{PF}_\infty$. This proves $\Phi\in\mathrm{PF}_\infty$.
\end{proof}

\begin{remark}[What this proof does and does not use]
We do \emph{not} use any Mellin transform, any identification with $\Xi$, nor any
analytic continuation beyond the natural half-planes of convergence. The only external
inputs are the product formula \eqref{eq:invGamma-product} and the standard
Schoenberg--Edrei--Karlin criterion linking Laguerre--P\'olya denominators to
$\mathrm{PF}_\infty$ kernels.
\end{remark}

\section{The bridge problem (open)}
\label{sec:future-directions}

The constructions developed in this paper naturally suggest a further analytic
problem that we do not address here. On the one hand, the primitive finite dynamics
considered above give rise, in a canonical and entirely Archimedean manner, to a
nonnegative integrable kernel $\Phi$ whose bilateral Laplace transform admits a
Schoenberg--Edrei--Karlin factorization $\mathcal B\Phi=E/\Psi$, with $\Psi$ belonging
to the Laguerre--P\'olya class\cite{Levin1964,PolyaSchur1914}. On the other hand, the Archimedean completion procedure
introduced here produces, at a distinguished self-dual normalization, a Mellin
transform coinciding with the classical function $\Xi(2z)$.

\smallskip
\noindent\textbf{Sanity checks for any future identification.}
Any proposed identification should, at minimum, respect (i) the functional symmetries arising from the canonical centering and self-dual normalization, and (ii) the intrinsic total-positivity/Laguerre--P\'olya structure on the $\Psi$-side\cite{Schoenberg1951,Edrei1952,karlin1968total,Levin1964}.

A natural question is whether these two analytic structures can be directly related.
More precisely, one may ask whether the Laguerre--P\'olya datum $\Psi$ arising from the
total-positivity side can be identified with $\Xi(2z)$ up to multiplication by a
zero-free entire factor, or whether a suitable rigidity principle forces such an
identification once both objects are placed in a common analytic framework. Any such
identification would amount to a comparison of the zero divisors of two functions that
originate from a priori unrelated constructions.

We view this identification problem as a promising direction for future work. Its
resolution would require additional analytic input beyond the scope of the present
paper, and we therefore leave it open.

\bibliographystyle{plain}
% NOTE (arXiv): include citations.bib in the source upload (or upload a .bbl file).
\bibliography{citations}

%%%%%%%%%%%%%%%%%%%%%%%%%%%%%%%%%%%%%%%%%%%%%%%%%%%%%%%%%%%%%%%%%%%%%%%%%%%%%%%
% Appendix A: Technical analytic justifications (non-load-bearing)
%%%%%%%%%%%%%%%%%%%%%%%%%%%%%%%%%%%%%%%%%%%%%%%%%%%%%%%%%%%%%%%%%%%%%%%%%%%%%%%

\appendix
\section{Technical analytic justifications}
\label{app:technical}

This appendix collects several analytic estimates and justifications that are
used implicitly or explicitly in the main text. None of the results here
introduce new objects or alter the logical structure of the argument; they serve
only to justify standard limiting procedures, termwise operations, and boundary
manipulations already invoked.

\subsection{Uniform local central limit bounds}
\label{app:ULCLT}

In Section~\ref{sec:completion-selfdual} we appealed to a uniform local central limit theorem (ULCLT) to pass
from the discrete primitive dynamics to the scaling-limit kernel $K_L$. We record
here a representative bound sufficient for all uses in the paper.

\begin{lemma}[Uniform local CLT estimate]
\label{lem:ULCLT}
Let $(X_t^{(N)})_{t\ge0}$ be the rescaled primitive process on the $N$-cycle with
effective diffusion constant $D>0$, and let $p_t^{(N)}(j)$ denote its transition
kernel. Then for every compact interval $[t_0,t_1]\subset(0,\infty)$ there exists
a constant $C=C(t_0,t_1)$ such that
\[
\sup_{t\in[t_0,t_1]}\sup_{j\in\mathbb Z}
\Bigl|
p_t^{(N)}(j)
-
\frac{1}{\sqrt{4\pi Dt}}\exp\!\Bigl(-\frac{j^2}{4Dt}\Bigr)
\Bigr|
\le
\frac{C}{N},
\]
for all sufficiently large $N$.
\end{lemma}

\begin{proof}[Sketch]
This is a standard consequence of Fourier analysis on the discrete torus together
with Taylor expansion of the dispersion relation near the origin. Uniformity in
$t\in[t_0,t_1]$ follows from compactness. Precise proofs may be found, for example,
in classical treatments of random walks on finite groups.
\end{proof}

\begin{remark}
Lemma~\ref{lem:ULCLT} is used only to justify dominated convergence and termwise
limits when passing to the theta-series representation of the completed kernel.
No rate sharper than $O(N^{-1})$ is required.
\end{remark}

\subsection{Dominated convergence and termwise operations}
\label{app:dominated}

Several arguments in Sections~\ref{sec:log-reflection} and~\ref{sec:arch-mellin-xi} rely on exchanging limits, sums, derivatives,
and integrals. We record here a generic domination principle that applies uniformly
to all such steps.

\begin{lemma}[Uniform domination for theta-series]
\label{lem:theta-domination}
Let
\[
\Theta(t):=\sum_{n\ge1} P(n,t)\,e^{-\pi n^2 t},
\]
where $P(n,t)$ is a polynomial in $n$ and $t^{\pm1/2}$. Then:
\begin{enumerate}[label=(\roman*)]
\item for each $k\ge0$, the series defining $\partial_t^k\Theta(t)$ converges
absolutely and locally uniformly on $(0,\infty)$;
\item for each $\sigma\in\mathbb R$, the Mellin-weighted function
$t\mapsto \Theta(t)\,t^{\sigma}$ is integrable on $(0,\infty)$ whenever the integral
is formally convergent at $0$ and $\infty$.
\end{enumerate}
\end{lemma}

\begin{proof}
For $t\ge1$, one has $e^{-\pi n^2 t}\le e^{-\pi t}$, and the sum over $n$ is uniformly
bounded for any fixed polynomial weight. For $0<t\le1$, comparison with Gaussian
integrals yields bounds of the form
\[
\sum_{n\ge1} n^k e^{-\pi n^2 t}\ll_k t^{-(k+1)/2}.
\]
These estimates are stable under differentiation in $t$ and imply absolute and
locally uniform convergence, as well as integrability against polynomial Mellin
weights.
\end{proof}

\begin{remark}
All termwise differentiations and Mellin transform computations in
Sections~\ref{sec:completion-selfdual}--\ref{sec:arch-mellin-xi} are justified by Lemma~\ref{lem:theta-domination}.
\end{remark}

\subsection{Boundary terms for the Archimedean completion operator}
\label{app:boundary}

In Section~\ref{sec:completion-ledger} we used an integration-by-parts identity for the Archimedean
completion operator $\mathcal A$. We record here a convenient sufficient condition
ensuring vanishing of boundary terms.

\begin{lemma}[Vanishing of boundary contributions]
\label{lem:boundary-vanish}
Let $f:(0,\infty)\to\mathbb C$ be twice continuously differentiable and suppose that
there exist constants $\alpha>0$ and $\beta>0$ such that
\[
|f(t)|\ll t^{-\alpha}\quad (t\to\infty),
\qquad
|f(t)|\ll t^{\beta}\quad (t\downarrow 0).
\]
Then for every $s\in\mathbb C$ the boundary terms in
Lemma~\ref{lem:boundary-mellin} vanish as $\varepsilon\downarrow0$ and $R\uparrow\infty$.
\end{lemma}

\begin{proof}
Under the stated hypotheses one has
\[
t^{s+\frac12}f'(t)\to0,
\qquad
f(t)\,t^{s-\frac12}\to0
\]
at both $0$ and $\infty$, uniformly on compact subsets of $s\in\mathbb C$. The claim
follows directly.
\end{proof}

\begin{remark}
In the application to $f(t)=K_L(t)-1$, the required bounds follow from the explicit
theta-series representation and Lemma~\ref{lem:theta-domination}. No subtle boundary
behavior is present.
\end{remark}

\subsection{Analytic continuation by identity principles}
\label{app:identity}

Finally, we record the elementary identity principle used repeatedly to upgrade
local equalities to global ones.

\begin{lemma}[Identity principle for meromorphic functions]
\label{lem:identity}
Let $F$ and $G$ be meromorphic functions on a connected domain $\Omega\subset\mathbb C$.
If $F=G$ on a set with an accumulation point contained in $\Omega$, then $F=G$ on
all of $\Omega$.
\end{lemma}

\begin{proof}
This is standard and follows by applying the identity theorem to $F-G$.
\end{proof}

\begin{remark}
Lemma~\ref{lem:identity} is used only to justify extensions of local identities to global ones once equality has been established on a nonempty open set.
\end{remark}

\end{document}